\documentclass{amsart}
\pdfoutput=1

\theoremstyle{plain}
\usepackage{amssymb,amsmath,enumitem,refcount,tikz,tikz-cd}

\usepackage{color}

\newtheorem{theorem}{Theorem}
\newtheorem{lemma}[theorem]{Lemma}
\newtheorem{proposition}[theorem]{Proposition}

\usetikzlibrary{calc}
\usetikzlibrary{decorations.pathmorphing}
\tikzset{curve/.style={settings={#1},to path={(\tikztostart)
    .. controls ($(\tikztostart)!\pv{pos}!(\tikztotarget)!\pv{height}!270:(\tikztotarget)$)
    and ($(\tikztostart)!1-\pv{pos}!(\tikztotarget)!\pv{height}!270:(\tikztotarget)$)
    .. (\tikztotarget)\tikztonodes}},
    settings/.code={\tikzset{quiver/.cd,#1}
        \def\pv##1{\pgfkeysvalueof{/tikz/quiver/##1}}},
    quiver/.cd,pos/.initial=0.35,height/.initial=0}
\tikzset{halo/.style={
         preaction={draw,white,line width=4pt,-},
         preaction={draw,white,ultra thick, shorten >=-2.5\pgflinewidth}}}
\makeatletter
\newcommand{\dotminus}{\mathbin{\text{\@dotminus}}}
\newcommand{\@dotminus}{\ooalign{\hidewidth\raise1ex\hbox{.}\hidewidth\cr$\m@th-$\cr}}
\makeatother

\newcommand{\cat}{\mathbf}

\newcommand{\subsetof}{\subseteq}
\newcommand{\id}{\mathrm{id}}
\newcommand{\inprod}[2]{\langle #1 \mid #2 \rangle}
\newcommand{\inl}{\mathrm{inl}}
\newcommand{\inr}{\mathrm{inr}}
\newcommand{\CC}{\mathbb{C}}
\newcommand{\RR}{\mathbb{R}}
\renewcommand{\H}{\mathcal H}
\newcommand{\K}{\mathcal K}

\newcommand{\NN}{\mathbb N}
\DeclareMathOperator{\coker}{coker}
\newcommand{\smult}{\bullet}
\newcommand{\comp}{\circ}

\theoremstyle{definition}
\newtheorem{example}[theorem]{Example}
\newtheorem{definition}[theorem]{Definition}
\theoremstyle{remark}

\begin{document}

\title[Axioms for the category of Hilbert spaces and linear contractions]{Axioms for the category of Hilbert spaces\\ and linear contractions}

\author{Chris Heunen}
\address{University of Edinburgh}
\email{chris.heunen@ed.ac.uk}

\author{Andre Kornell}
\address{Dalhousie University}
\email{akornell@dal.ca}

\author{Nesta van der Schaaf}
\address{University of Edinburgh}
\email{n.schaaf@ed.ac.uk}

\thanks{Andre Kornell was supported by the Air Force Office of Scientific Research under Awards No.\ FA9550-16-1-0082 and FA9550-21-1-0041.}

\begin{abstract}
  The category of Hilbert spaces and linear contractions is characterised by elementary categorical properties that do not refer to probabilities, complex numbers, norm, continuity, convexity, or dimension.  
\end{abstract}

\maketitle

\section{Introduction}

Quantum mechanics, in its traditional formulation, is based on Hilbert space~\cite{vonneumann:foundations}. More precisely, it is based on mappings of Hilbert spaces. In the simplest telling, states are unit vectors, evolving along certain linear maps between Hilbert spaces. Taking into account probabilities, states become operators, and dynamics certain linear maps between the resulting algebras. Either way, it is fair to say that the category of Hilbert spaces and linear maps is crucial to quantum theory. 

This model gives accurate results, but the framework of Hilbert spaces is hard to interpret physically from first principles~\cite{redei:vonneumannhilbert}.
Many reconstruction programmes try to reformulate the framework mathematically, so that the primitive assumptions rely on fewer physically unjustified details~\cite{grinbaum:reconstruction}. For example, instead of starting with Hilbert space, they start from operator algebras~\cite{murrayneumann:ringsofoperators}, orthomodular lattices~\cite{piron:foundations}, generalised probabilistic theories~\cite{segal:postulates}, or from categories~\cite{heunenvicary}.

In this final case, the situation comes down to the following. Somebody hands you a category. How do you know it is (equivalent to) that of Hilbert spaces? The answer depends on which morphisms between Hilbert spaces you choose exactly. Previous work~\cite{heunenkornell:hilbert} settles the case of bounded linear maps, by giving axioms characterising that category precisely. Importantly, those axioms are purely categorical, and do not mention probabilities, complex numbers, or other analytical details that you may think are fundamental to Hilbert space.

Nevertheless, the story does not end with that answer, because not all bounded linear maps are physical. Quantum systems may only evolve along unitary linear maps. Taking quantum measurement into account allows the larger class of linear contractions~\cite{selinger:flowcharts,andresmartinezheunenkaarsgaard:cptn}. 
But no quantum-theoretical process can be described by a bounded linear map that is not a contraction.
This article settles the case for the physical choice of morphisms being linear contractions. 

The main idea is the following: The subcategory of Hilbert spaces and linear contractions generates the category of Hilbert spaces and all bounded linear maps as a monoidal category with invertible nonzero scalars (where scalars in a monoidal category are endomorphism on the tensor unit). Specifically, if you formally adjoin inverses for all nonzero scalars to the former category, you get the latter. We find properties of the former category that guarantee that its completion satisfies the axioms of~\cite{heunenkornell:hilbert} and hence is the category of Hilbert spaces and bounded linear maps. Finally, using the nature of the completion, we show that the former category must consist of exactly the linear contractions. 

The presented axioms are natural from a categorical point of view. They progress in a modular way, so that fragments are compatible with other research programmes. For example, biproducts and equalisers are used as in abelian categories and regular categories~\cite{borceux:2}, the tensor product is used as in categorical quantum mechanics~\cite{heunenvicary} and the string diagrams of ZX-calculus~\cite{coeckeduncan:zx}, and the combination of daggers and kernels is used as in quantum logic~\cite{heunenjacobs:kernels}. Overall, our axioms begin with the structure of dagger rig categories, which form the starting point for a complete graphical model of mixed state quantum theory~\cite{carettejeandelperdrixvilmart:completeness} as well as for both classical and quantum reversible programming~\cite{heunenkaarsgaard:quantumeffects,caretteheunenkaarsgaardsabry:roots}. Moreover, these axioms enable categorical methods such as the use of universal properties in quantum information theory~\cite{huotstaton:channels}.

In addition to being an important characterisation in its own right, this theorem is also a stepping stone towards trying to characterise related categories:
Hilbert spaces and completely positive morphisms, going towards foundations of quantum information theory;
Hilbert spaces and unitaries, going towards foundations of quantum computing;
and Hilbert modules and adjointable morphisms, going towards unitary representations and foundations of quantum field theory~\cite{barbosaheunen:sheaves}.

The rest of this article is structured as follows. We start by stating the axioms in Section~\ref{sec:axioms}. Their meaning is discussed in Section~\ref{sec:category}. In particular, the category of Hilbert spaces and linear contractions is defined there, and it is shown how it satisfies the axioms. Next, Section~\ref{sec:basics} derives basic properties from the axioms. The real work begins in Section~\ref{sec:completion}, which defines a completion by which we can abstractly recognise linear contractions among all bounded linear maps. Section~\ref{sec:theorem} puts everything together to prove the main theorem. 

\section{The axioms}\label{sec:axioms}

We will consider the following properties of a locally small category that is equipped with a contravariant endofunctor $\dagger$ and two symmetric monoidal structures $(\otimes,I)$ and $(\oplus,0)$:

\begin{enumerate}[label=(\arabic*),ref=\arabic*]
  \item\label{axiom:dagger}
  The contravariant endofunctor $\dagger$ satisfies $\id_H^\dagger = \id_H$ for all objects $H$ and $t^{\dagger \dagger} = t$ for all morphisms $t$.
  A category equipped with such a dagger is also called a \emph{dagger category}.
  If $t^\dag \comp t = \id$, we call $t$ a \emph{dagger monomorphism}, and if additionally $t \comp t^\dag = \id$, we call it a \emph{dagger isomorphism}.

  \item\label{axiom:rig}
  The category is a \emph{dagger rig} category: the unitors, braidings, and associators of the monoidal structures $(\otimes,I)$ and $(\oplus,0)$ are dagger isomorphisms, and they satisfy $(s \otimes t)^\dag=s^\dag \otimes t^\dag$ and $(s \oplus t)^\dag=s^\dag \oplus t^\dag$ for all morphisms $s$ and $t$, and there are natural dagger isomorphisms
  \begin{align*}
    H \otimes (K \oplus L) & \to (H \otimes K) \oplus (H \otimes L) \\
    (H \oplus K) \otimes L & \to (H \otimes L) \oplus (K \otimes L) 
  \end{align*}
  that satisfy the coherence conditions (I)--(XXIV) of~\cite{laplaza:rig}.

  \item\label{axiom:affine}
  The unit $0$ is initial, and thus a zero object. We write $0_{H,K}$ for the unique morphism $H \to K$ that factors through the object $0$. It follows that there are natural morphisms:
  \begin{align*}
    \inl_{H,K} & = \Big(\smash{\begin{aligned}\begin{tikzpicture}[xscale=5]
      \node (h0) at (1,0) {$H \simeq H \oplus 0$};
      \node (hk) at (2,0) {$H \oplus K$};
      \draw[->] (h0) to node[above]{$\id \oplus 0$} (hk);
    \end{tikzpicture}\end{aligned}}\Big)
    \\
    \inr_{H,K} & = \Big(\smash{\begin{aligned}\begin{tikzpicture}[xscale=5]
      \node (h0) at (1,0) {$K \simeq 0 \oplus K$};
      \node (hk) at (2,0) {$H \oplus K$};
      \draw[->] (h0) to node[above]{$0 \oplus \id$} (hk);
    \end{tikzpicture}\end{aligned}}\Big)
  \end{align*}
  This property of $\oplus$ makes the category so-called \emph{affine} or \emph{semicartesian}.

  \item\label{axiom:jointlyepic}
  The injections $\inl_{H,K}$ and $\inr_{H,K}$ are \emph{jointly epic}:
  $s=t \colon H \oplus K \to L$ as soon as $s \comp \inl = t \comp \inl$ and $s \comp \inr = t \comp \inr$.

  \item\label{axiom:nondegenerate}
  \emph{Mixture} occurs: there exists a morphism $s \colon I \to I \oplus I$ with $\inl^\dag \comp s \neq 0 \neq \inr^\dag \comp s$.

  \item\label{axiom:simple}
  The unit $I$ is \emph{dagger simple}: any dagger monomorphism $S \to I$ is either invertible or zero but not both.

  \item\label{axiom:separator}
  The unit $I$ is a \emph{monoidal separator}: $s=t \colon H \otimes K \to L$ as soon as $s \comp (x \otimes y) = t \comp (x \otimes y)$ for all $x \colon I \to H$ and $y \colon I \to K$.

  \item\label{axiom:equalisers}
  Any two parallel morphisms have a \emph{dagger equaliser}, that is, an equaliser that is a dagger monomorphism.

  \item\label{axiom:kernels}
  Any dagger monomorphism $N \to H$ is a \emph{dagger kernel}, that is, a dagger equaliser of some morphism $H \to K$ and $0 \colon H \to K$. 

  \item\label{axiom:positive}
  Subobjects are determined by \emph{positive} maps: monomorphisms $r \colon H \rightarrowtail L$ and $s \colon K \rightarrowtail L$ satisfy $r=s \comp t$ for some isomorphism $t$ if and only if $r \comp r^\dag = s \comp s^\dag$.

  \item\label{axiom:colimits}
  Any directed diagram has a colimit.
\end{enumerate}
\setcounter{equation}{\theenumi}

Observe that these axioms are all elementary properties of categories: none of them refer to notions such as complex amplitudes, probabilities, norm, convexity, continuity, or metric completeness. 
These are natural axioms from a categorical point of view, and thus we hope that they can be interpreted as natural assumptions about quantum information processes. 
For example, one might interpret axiom~\eqref{axiom:dagger} in terms of time reversal symmetry and axioms~\eqref{axiom:rig}-\eqref{axiom:nondegenerate} in terms of superposition.
However, this task is beyond the scope of this work, and we leave it to future research. Furthermore, we conjecture that these axioms are all necessary for our result and independent. For example, the category of sets and partial injections satisfies all the axioms except~\eqref{axiom:nondegenerate}.

We end this section with a brief summary of the main ideas of previous work~\cite{heunenkornell:hilbert} that we will rely on.
That work assumed similar axioms as here. Axioms~\eqref{axiom:dagger} and \eqref{axiom:separator}--\eqref{axiom:kernels} are literally the same. Axioms~\eqref{axiom:rig}--\eqref{axiom:nondegenerate} resemble the axiom there stipulating dagger biproducts. Axiom~\eqref{axiom:simple} here concerns dagger monomorphisms but there concerns all monomorphisms, and similarly, axiom~\eqref{axiom:colimits} here concerns colimits in the category itself but there concerns colimits in the wide subcategory of dagger monomorphisms. Axiom~\eqref{axiom:positive} here has no analogue there. A category $\cat{C}$ satisfying those axioms is there shown to be equivalent to the category of Hilbert spaces as follows.
Just as vectors of a Hilbert space~$\H$ are recovered as linear maps $\mathbb{C}\to \H$ in the complex case, the idea is that an object $H\in\cat{C}$ can be turned into a concrete Hilbert space by taking the set $\cat{C}(I,H)$, which intuitively consists of the vectors of $H$, and equipping it with the structure of a complete inner product space over the field $\cat{C}(I,I)$. The endomorphisms $I\to I$ always form a commutative monoid, but due to the axioms imposed on $\cat{C}$ it turns into an involutive field~\cite[Lemma~1]{heunenkornell:hilbert}. In fact, by Sol\`er's Theorem, this field must be either $\mathbb{R}$ or $\mathbb{C}$ \cite[Proposition~5]{heunenkornell:hilbert}. The main result then showed that there is a functor $\cat{C}(I,-)$ from $\cat{C}$ to the category of Hilbert spaces that defines an equivalence of dagger rig categories \cite[Theorem~8]{heunenkornell:hilbert}.

\section{The category}\label{sec:category}

We define our main model. Recall that a \emph{Hilbert space} is a vector space $\mathcal{H}$ over $\mathbb{R}$ or $\mathbb{C}$, together with an inner product $\inprod{-}{-}$ whose induced norm $\| x\|^2= \langle x\vert x\rangle$ makes $\mathcal{H}$ Cauchy-complete.

\begin{definition}
  A linear function $T \colon \H \to \K$ between Hilbert spaces is \emph{bounded} when there exists a constant $N \in [0,\infty)$ such that $\|Tx\| \leq N \cdot \|x\|$ for all $x \in \H$; write $\|T\|$ for the infimum of such $N$.
  The function $T$ is called a \emph{contraction} when $\|T\|\leq 1$, that is, when $\|Tx\| \leq \|x\|$ for all $x \in \H$. Such linear functions are also known as \emph{nonexpansive} or \emph{short} maps.
  We write $\cat{Hilb}_\RR$ and $\cat{Hilb}_\CC$ for the categories of Hilbert spaces and bounded linear maps, and $\cat{Con}_\RR$ and $\cat{Con}_\CC$ for the categories of Hilbert spaces and short linear maps, respectively over the real and complex numbers. When the distinction does not matter, we will simply write $\cat{Hilb}$ and $\cat{Con}$. 
\end{definition}

Let us discuss how $\cat{Con}$ satisfies the axioms.

\begin{enumerate}
  \item[\eqref{axiom:dagger}]    
  The dagger is provided by adjoints: for any bounded linear function $T \colon \H \to \K$ there is a unique linear map $T^\dag \colon \K \to \H$ satisfying ${\inprod{Tx}{y}_\K = \inprod{x}{T^\dag y}_\H}$ for all $x \in \H$ and $y \in \K$. This is a well-defined functor on $\cat{Con}$ because $\|T^\dag\| = \|T\|$. The dagger monomorphisms are the isometries, and the dagger isomorphisms are the unitaries. We shall use both terminologies, depending on whether we are working with an abstract dagger category or are in the concrete setting of Hilbert spaces.

  \item[\eqref{axiom:rig}]
  The monoidal structures $\otimes$ and $\oplus$ are produced by tensor products and direct sums of Hilbert spaces, respectively. This is well-defined in $\cat{Con}$ because $\|S \otimes T\| = \|S\| \cdot \|T\|$ and $\|S \oplus T\|=\max\{\|S\|,\|T\|\}$.
  The unit $I$ is the one-dimensional Hilbert space given by the base field, and the unit $0$ is the zero-dimensional Hilbert space $\{0\}$.

  \item[\eqref{axiom:affine}]
  Linear maps must preserve the 0 vector, so the object $0$ is indeed initial.

  \item[\eqref{axiom:jointlyepic}]
  The injections $\inl_{\H,\K}(x)=(x,0)$ and $\inr_{\H,\K}(y)=(0,y)$ have images that linearly span $\H \oplus \K$, so are clearly jointly epic.

  \item[\eqref{axiom:nondegenerate}]
  The vector $s=(\tfrac{1}{2},\tfrac{1}{2})$ in the two-dimensional Hilbert space satisfies~\eqref{axiom:nondegenerate}, so this axiom may be read as saying that nontrivial superpositions of qubits exist.

  \item[\eqref{axiom:simple}]
  The only vector subspaces of the base field are the zero-dimensional one and the whole space itself.

  \item[\eqref{axiom:separator}] 
  Linear contractions are continuous, and the linear span of pure tensor elements $x \otimes y$ is dense in $\H \otimes \K$, so morphisms in $\cat{Con}$ out of a tensor product are indeed determined by their action on pure tensor elements.

  \item[\eqref{axiom:equalisers}]
  If $S,T \colon \H \to \K$ are linear contractions, then $\{x \in \H \mid Sx=Tx\}$ is a closed subspace of $\H$ and hence an object of $\cat{Con}$, and the inclusion is automatically an isometry.

  \item[\eqref{axiom:kernels}]
  Any closed subspace $\mathcal{N} \subseteq \H$ is a kernel of the orthogonal projection $\inr^\dag \colon \H \simeq \mathcal{N} \oplus \mathcal{N}^\perp \to \mathcal{N}^\perp$ onto its orthocomplement, which is a contraction.

\end{enumerate}

The last two axioms are perhaps a bit less standard, and we spell out proofs.

\begin{lemma}
  The category $\cat{Con}$ satisfies~\eqref{axiom:positive}.
\end{lemma}
\begin{proof}
  Observe that any isomorphism in $\cat{Con}$ is unitary: if $\|T\| \leq 1$ and $\|T^{-1}\| \leq 1$, then $\|x\|=\|T^{-1}(T(x))\| \leq \|Tx\| \leq \|x\|$ for all vectors $x$, so $T$ is an isometry and hence unitary.
  This makes one direction of~\eqref{axiom:positive} clear: if $R=S\comp T$ for unitary $T$, then $R\comp R^\dag = S\comp T\comp T^\dag\comp S^\dag = S\comp S^\dag$.
  The other direction follows directly from Douglas' Lemma~\cite{douglas1966MajorizationFactorizationRange}.
\end{proof}

\begin{lemma}
  The category $\cat{Con}$ satisfies~\eqref{axiom:colimits}.
\end{lemma}
\begin{proof}
  See also~\cite[Proposition~3.1]{nikolicdipierro:colimits}.
  Let $(J,\leq)$ be a directed partially ordered set, and $\H_{i \leq j} \colon \H_i \to \H_j$ be a diagram $J \to \cat{Con}$.
  Let $V$ be the vector space $\bigcup_{j \in J} \H_j$, where $x \in \H_i$ and $y \in H_j$ are identified when $\H_{i \leq k}(x)=\H_{j \leq k}(y)$ for some $k \in J$.
  Writing $[x_i]$ for the equivalence class of $x_i \in \H_i$, define $\|[x_i]\| = \lim_j \|\H_{i \leq j}(x_i)\|$; it is routine to verify that this limit exists and defines a seminorm on $V$. (Recall that a seminorm is a function $p \colon V \to \RR$ that satisfies all properties of a norm except positive definiteness; so $p(x)=0$ need not imply $x=0$.)
  This seminorm satisfies the parallelogram law. So if $N = \{ x \in V \mid \|x\|= 0 \}$, then $V/N$ is an inner product space. Call its completion $\H_\infty$, define $\H_{j < \infty} \colon \H_j \to \H_\infty$ by $x_j \mapsto [x_j]$, and observe that this is a cocone.
  
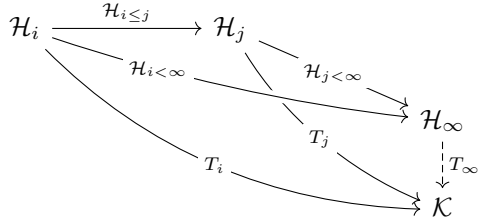
\[\begin{tikzcd}[column sep=2cm]
    {\H_i} & {\H_j} \\
    && {\H_\infty} \\
    && \K
    \arrow["{\H_{i \leq j}}", from=1-1, to=1-2]
    \arrow["{T_\infty}", dashed, from=2-3, to=3-3]
    \arrow["{T_j}"{description}, curve={height=12pt}, from=1-2, to=3-3]
    \arrow["{T_i}"{description}, curve={height=24pt}, from=1-1, to=3-3]
    \arrow["{\H_{i<\infty}}"{description,pos=0.3},halo,curve={height=6pt}, from=1-1, to=2-3]
    \arrow["{\H_{j<\infty}}"{description}, from=1-2, to=2-3]
  \end{tikzcd}\]
  Suppose that $T_j \colon \H_j \to \K$ is another cocone.
  It defines a function $\bigcup_{j \in J} \H_j \to \K$ that respects the equivalence, and so lifts to a linear function $T \colon V \to \K$.
  Furthermore, if $x_j \in \H_j$ satisfies $\|[x_j]\|=0$, then $\|T[x_j]\|=\|T[\H_{j \leq k}(x_j)]\| \leq \|\H_{j \leq k}(x_j)\| \to 0$. Thus $T$ lifts to a linear function $V/N \to \K$. Similarly, this function is a contraction, and so extends to a linear contraction $T_\infty \colon \H_\infty \to \K$.
  By construction $T_\infty(\H_{j<\infty}(x_j))=T_\infty[x_j]=T_j(x_j)$ for all $x_j \in \H_j$, and so $T_\infty$ is a mediating map from the cocone $\{\H_{j<\infty}\}_{j \in J}$ to the cocone $\{T_j\}_{j \in J}$. Finally, this mediating map is unique, because the $\H_{j<\infty}$ have jointly dense range in $\H_\infty$ by construction.
\end{proof}

Axioms~\eqref{axiom:positive} and~\eqref{axiom:colimits} are the only ones that hold in $\cat{Con}$ but not in $\cat{Hilb}$. The behaviour following from these two axioms accounts for the difference between the two categories.

We end this section by determining the subobjects of $I$ in $\cat{Con}$.
Contrast the following lemma to the situation $\cat{Hilb}$, where the subobjects of $I$ are $\{0,1\}$; the crucial difference is that scalars $0<z<1$ are invertible in $\cat{Hilb}$ but not in $\cat{Con}$.

\begin{lemma}\label{lem:scalarsinCon}
  There is an order isomorphism between the subobjects of $I$ in $\cat{Con}$ and the real unit interval $[0,1]$.
\end{lemma}
\begin{proof}
  Consider a monomorphism of $\cat{Con}$ into the base field $I$.
  It is an injective linear contraction $T \colon \H \to I$. The kernel of $T$ is zero, and so $\dim \H = 0$ or $\dim \H = 1$. Hence a subobject of $I$ is represented by the unique morphism $0 \to I$ or by an injective contraction $I \to I$. Up to isomorphism of subobjects, the latter morphisms are scalars in the interval $(0,1]$. Hence the subobjects of $I$ in $\cat{Con}$ are canonically in bijection with the closed unit interval $[0,1]$. This bijection is clearly an order isomorphism.     
\end{proof}

\section{The basic lemmas}\label{sec:basics}

From now on, we assume a (locally small) category $\cat{D}$ that satisfies the axioms~\eqref{axiom:dagger}--\eqref{axiom:colimits}, and set out to prove that $\cat{D} \simeq \cat{Con}$. 
In this section we derive some basic properties of $\cat{D}$.
We start by recalling a factorisation that already follows from~\eqref{axiom:equalisers} and~\eqref{axiom:kernels} alone~\cite{heunenjacobs:kernels}. We summarise a proof here for convenience.

\begin{lemma}\label{lem:factorisation}
  Any morphism $t \colon H \to K$ factors as an epimorphism $e \colon H \to E$ followed by a dagger monomorphism $k=\ker(\ker(t^\dag)^\dag) \colon E \to K$.
\end{lemma}
\begin{proof}
  By~\eqref{axiom:dagger}, $\ker(t^\dag)^\dag$ is a cokernel of $t$, so that $\ker(t^\dag)^\dag \comp t=0$, and $t$ always factors through $k$ via some morphism $e$. It remains to show that $e$ is an epimorphism.
  
  We first prove that if $s \comp e=0$, then $s=0 \colon E \to L$. Observe that $t^\dag \comp k \comp s^\dag = e^\dag \comp s^\dag = 0$. Hence $k \comp s^\dag$ factors through $\ker(t^\dag)$ via some $r \colon L \to F$. 
  \[\begin{tikzpicture}[xscale=4,yscale=1.5]
    \node (f) at (0,1) {$\smult$};
    \node (k) at (1,1) {$K$};
    \node (h) at (2,1) {$H$};
    \node (l) at (0,0) {$L$};
    \draw[->] (f) to node[above]{$\ker(t^\dag)$} (k);
    \draw[->] ([yshift=.5mm]k.east) to node[above]{$t^\dag$} ([yshift=.5mm]h.west);
    \draw[->] ([yshift=-.5mm]k.east) to node[below]{$0$} ([yshift=-.5mm]h.west); 
    \draw[->] (l) to node[below]{$k \comp s^\dag$} (k);
    \draw[->,dashed] (l) to node[left]{$r$} (f);
  \end{tikzpicture}\]
  But $r=\ker(t^\dag)^\dag \comp k \comp s^\dag = 0$, so $k \comp s^\dag = 0$, and so $s=0$.

  Next, we prove that $e$ is an epimorphism. Suppose that $s \comp e = s' \comp e$ for morphisms $s,s' \colon E \to L$. Let $m$ be a dagger equaliser of $s$ and $s'$. Then $e$ factors through $m$, and, writing $\coker(m)$ for a cokernel of $m$, we infer $\coker(m) \comp e = 0$.
  \[\begin{tikzpicture}[xscale=4,yscale=1.5]
    \node (h) at (0,1) {$H$};
    \node (e) at (1,1) {$E$};
    \node (l) at (2,1) {$L$};
    \node (f) at (0,0) {$\smult$};
    \node (g) at (1,0) {$\smult$};
    \draw[->] (h) to node[above]{$e$} (e);
    \draw[->] ([yshift=.5mm]e.east) to node[above]{$s$} ([yshift=.5mm]l.west);
    \draw[->] ([yshift=-.5mm]e.east) to node[below]{$s'$} ([yshift=-.5mm]l.west);
    \draw[->] (f) to node[below]{$m$} (e);
    \draw[->,dashed] (h) to (f);
    \draw[->] (e) to node[right]{$\coker(m)$} (g);
  \end{tikzpicture}\]
  But by the property we proved earlier, that means $\coker(m)=0$, and so $m$ is invertible. We conclude that $s=s'$.
\end{proof}

Next we notice a consequence of axiom~\eqref{axiom:simple}.
A \emph{scalar} in a monoidal category is a morphism $z \colon I \to I$, and we can multiply any morphism $t \colon H \to K$ by it to obtain a morphism $H \simeq I \otimes H \stackrel{z \otimes t}{\longrightarrow} I \otimes K \simeq K$ that we denote $z \smult t$~\cite[2.1]{heunenvicary}. 
For a scalar $z \colon I \to I$ and morphisms $s \colon H \to K$ and $t \colon K \to L$, it follows from the bifunctoriality of the tensor product that $z \smult (g \comp f) = (z \smult g)\comp f = g \comp (z \smult f)$. In particular, $w \smult z = w \comp z$ for scalars $w,z \colon I \to I$. Instead of $\id_I$ we may write $1$ for the identity scalar.

\begin{lemma}\label{lem:scalarsmonic}
  Every nonzero scalar is monic and epic.
\end{lemma}
\begin{proof}
  Let $z \colon I \to I$ be nonzero.
  Lemma~\ref{lem:factorisation} factors it as an epimorphism followed by a dagger monomorphism. Axiom~\eqref{axiom:simple} guarantees that that dagger monomorphism is either 0 or invertible. By assumption it cannot be zero, and so $z$ is epic. Similarly, $z^\dag$ is epic, and so $z$ is also monic.
\end{proof}

It follows that every dagger monic scalar is invertible, for it cannot be zero, so must be epic as well as split monic.

\begin{lemma}\label{lem:scalarmultiplicationmonic}
 For all nonzero scalars $z$ and morphisms $s,t \colon H \to K$, if $z \smult s = z \smult t$, then $s=t$.
\end{lemma}
\begin{proof}
 By two applications of axiom~\eqref{axiom:separator}, it suffices to prove that $y^\dag \comp s \comp x = y^\dag \comp t \comp x$ for all $x \colon I \to H$ and $y \colon I \to K$.
 But $z \comp (y^\dag \comp s \comp x) = z \smult (y^\dag \comp s \comp x) = y^\dag \comp (z \smult s) \comp x = y^\dag \comp (z \smult t) \comp x = z \smult (y^\dag \comp t \comp x) = z \comp (y^\dag \comp t \comp x)$, and the scalar $z$ is monic by Lemma~\ref{lem:scalarsmonic}.
\end{proof}

The following two lemmas are consequences of axiom~\eqref{axiom:positive}.

\begin{lemma}\label{lem:daggeriso}
  Any isomorphism is a dagger isomorphism. 
  Any split monomorphism is a dagger monomorphism.
\end{lemma}
\begin{proof}
  Let $r \colon H \to L$ be an isomorphism. Taking $s=\id_L$ and $t=r$ in~\eqref{axiom:positive} shows that $r \comp r^\dag = \id \comp \id^\dag = \id$. Hence $r^{-1} = r^{-1} \comp \id_L = r^{-1} \comp r \comp r^\dag = r^\dag$.

  Now suppose that $t \colon H \to K$ is a split monomorphism. Factor $t = k \comp e$ for a dagger monomorphism $k \colon E \to K$ and an epimorphism $e \colon H \to E$. Then the epimorphism $e$ is itself split monic. But in any category, a split monic epimorphism is an isomorphism, and hence $e$ is a dagger isomorphism. Therefore, $t$ is a dagger monomorphism.
\end{proof}

\begin{lemma}\label{lem:daggermono}
  If a scalar $z$ and a morphism $t\colon H \to K$ satisfy $t^\dag \comp t = z^\dag \smult z \smult \id_H$, then $t=z \smult s$ for a dagger monomorphism $s \colon H \to K$.
\end{lemma}
\begin{proof}
  Factor $t=k \comp e$ for a dagger monomorphism $k \colon E \to K$ and an epimorphism $e \colon H \to E$. Then $e^\dag \comp e = t^\dag\comp t = (z \smult \id_H)^\dag \comp (z \smult \id_H)$. Now~\eqref{axiom:positive} and Lemma~\ref{lem:daggeriso} provide a dagger isomorphism $u \colon H \to E$ such that $e=z \smult u$, and so $t=z \smult s$ for the dagger monomorphism $s=k \comp u$.   
\end{proof}

\section{The completion}\label{sec:completion}

This section contains the main construction of the proof. 
Lemmas~\ref{lem:scalarsinCon},~\ref{lem:scalarsmonic}, and~\ref{lem:daggeriso} show that scalars in $\cat{Con}$ correspond to the unit disc:
\[
  \cat{Con}_\RR(I,I) \simeq [-1,1]
  \qquad\qquad
  \cat{Con}_\CC(I,I) \simeq \{ z \in \mathbb{C} \mid |z|\leq 1\}
\]
Observe that nonzero scalars that are not on the unit circle have no inverse in $\cat{Con}$. But they do in $\cat{Hilb}$. In fact, $\cat{Hilb}$ is the localisation of $\cat{Con}$ at nonzero scalars. We now abstractly construct the localisation of a monoidal category at its nonzero scalars in a way that respects the axioms~\eqref{axiom:dagger}--\eqref{axiom:colimits}. Write $\mathcal{D}=\cat{D}(I,I)$ for the scalars of $\cat{D}$.

\begin{proposition}\label{prop:completion}
  There is a category $\cat{D}[\mathcal{D}^{-1}]$ with the same objects as $\cat{D}$, where a morphism $[t/z]$ consists of a nonzero scalar $z$ and a morphism $t$ in $\cat{D}$ modulo the following equivalence relation:
  \[
    [t/z]\sim [t'/z'] \iff z' \smult t = z \smult t'
  \]
  The identity on $H$ is $[\id_H/1]$, and composition is given by:
  \[
    [t/z] \comp [s/w] = [t \comp s / z \smult w]
  \]
  There is an embedding $\cat{D} \to \cat{D}[\mathcal{D}^{-1}]$ defined by $t \mapsto [t/1]$.
\end{proposition}
\begin{proof}
  The relation $\sim$ is clearly reflexive and symmetric. To see that it is transitive, suppose that $[r/x] \sim [s/y]$ and $[s/y] \sim [t/z]$. Then $r \smult y = s \smult x$ and $s \smult z = t \smult y$. It follows that $s \smult r \smult z = r \smult t \smult y = s \smult t \smult x$. Lemma~\ref{lem:scalarmultiplicationmonic} now shows that $r \smult z = t \smult x$, that is, $[r/x] \sim [t/z]$. Hence $\sim$ is an equivalence relation.

  To see that the composition is well-defined, suppose that $[t/z] \sim [t'/z']$ and $[s/w] \sim [s'/w']$. Then $z' \smult w' \smult (t \comp s) = (z' \smult t) \comp (w' \smult s) = (z \smult t') \comp (w \smult s') = z \smult w \smult (t' \comp s')$, so $[t/z] \comp [s/w] \sim [t'/z'] \comp [s'/w']$. It is clearly associative and satisfies the identity laws.
  If $z$ and $w$ are nonzero scalars, then $z \smult w = z \comp w$ is nonzero because $w$ is monic by Lemma \ref{lem:scalarsmonic}.

  The assignment $t \mapsto [t/1]$ is functorial, injective on objects, and faithful.
\end{proof}

The previous proposition concretely describes a quotient category~\cite[II.8]{maclane:categories}: writing $\cat{I}$ for the one-object category of non-zero scalars of $\cat{D}$, take the quotient of $\cat{D} \times \cat{I}$ under the congruence relation $(s,w) \sim (t,z) \iff z \bullet s = w \bullet t$.

\begin{example}
  There is an isomorphism $F \colon \cat{Con}[\mathcal{D}^{-1}] \to \cat{Hilb}$ of categories. It is defined as the identity $F(\H)=\H$ on objects, and as $F[T/z]=z^{-1}T$ on morphims. It is faithful by construction, because $F[S/w]=F[T/z]$ implies $zS=wT$ and so $[S/w]\sim [T/z]$. 
  To see that it is full, let $T \colon \H \to \K$ be a bounded linear function. If $\|T\| \leq 1$, then $T=F([T/1])$. 
  If $\|T\| \geq 1$, then $\|T\|^{-1} \in (0,1]$ is a scalar in $\cat{Con}$ and $T/\|T\|$ is a contraction, so $F\big(\big[ \|T\|^{-1} T / \|T\|^{-1}  \big]\big) = T$.
\end{example}

\begin{lemma}\label{lem:monoidal}
  The category $\cat{D}[\mathcal{D}^{-1}]$ inherits monoidal structure $\otimes$ from $\cat{D}$, and the functor $\cat{D} \to \cat{D}[\mathcal{D}^{-1}]$ is strict monoidal for $\otimes$.
\end{lemma}
\begin{proof}
  Observe that the equivalence relation of Proposition~\ref{prop:completion} is a monoidal congruence: if $[s/w] \sim [s'/w']$ and $[t/z] \sim [t'/z']$, then $[s \otimes t / w \smult z] \sim [s' \otimes t' / w' \smult z']$.
\end{proof}

Before we show that it satisfies the axioms of~\cite{heunenkornell:hilbert} one by one, let us first establish the universal property that justifies the notation $\cat{D}[\mathcal{D}^{-1}]$. Notice that any morphism $[t/z]$ in $\cat{D}[\mathcal{D}^{-1}]$ can be factored as $[t/z]=[\id/z] \comp [t/1]$.

\begin{proposition}\label{prop:universalproperty}
  Any functor $F \colon \cat{D} \to \cat{C}$ that is strong monoidal for $\otimes$ such that $F(z)$ is invertible for all nonzero scalars $z$ factors through a unique functor $\cat{D}[\mathcal{D}^{-1}] \to \cat{C}$ that is strong monoidal for $\otimes$ via the functor $\cat{D} \to \cat{D}[\mathcal{D}^{-1}]$.
  \[\begin{tikzpicture}[xscale=3,yscale=2]
    \node (c) at (0,1) {$\cat{D}$};
    \node (d) at (1,0) {$\cat{C}$};
    \node (bc) at (1,1) {$\cat{D}[\mathcal{D}^{-1}]$};
    \draw[->] (c) to (bc);
    \draw[->] (c) to node[below]{$F$} (d);
    \draw[->,dashed] (bc) to (d);
  \end{tikzpicture}\]
\end{proposition}
\begin{proof}
  Define the functor $G \colon \cat{D}[\mathcal{D}^{-1}] \to \cat{C}$ by $G(H)=F(H)$ on objects, and by $G[t/z]=F(z)^{-1} \smult F(t)$ on morphisms. This is the only functor making the triangle commute, because it is completely determined by its values at $[t/1]$ and $[\id/z]$ for all morphisms $t$ and all scalars $z$ in $\cat{D}$.
\end{proof}

Thus $\cat{D}[\mathcal{D}^{-1}]$ is the localisation of $\cat{D}$ at all nonzero scalars: it formally adjoins inverses for all nonzero scalars to $\cat{D}$. The concrete description of Proposition~\ref{prop:completion} simplifies the general construction for localisation~\cite{kashiwaraschapira:sheaves}. 

\begin{lemma}\label{lem:dagger}
  The category $\cat{D}[\mathcal{D}^{-1}]$ has a dagger, and the embedding $\cat{D} \to \cat{D}[\mathcal{D}^{-1}]$ preserves it. The embedding restricts to an isomorphism between the wide subcategories of dagger monomorphisms in $\cat{D}$ and $\cat{D}[\mathcal{D}^{-1}]$.
\end{lemma}
\begin{proof}
  Set $[t/z]^\dag = [t^\dag / z^\dag]$, which is clearly well-defined.
  By Proposition~\ref{prop:completion}, the embedding $\cat{D} \to \cat{D}[\mathcal{D}^{-1}]$ sends an object $A$ to itself, and sends a morphism $t$ to $[t/1]$. 
  It is faithful by Lemma~\ref{lem:scalarmultiplicationmonic}. 
  The embedding clearly preserves daggers, hence dagger monomorphisms, and so restricts to the wide subcategories of dagger monomorphisms. 
  This restricted functor is still bijective on objects and faithful. 
  To see that it is full, let $[t/z]$ be a dagger monomorphism in $\cat{D}[\mathcal{D}^{-1}]$. The assumption $[\id/1] = [t/z]^\dag \comp [t/z] = [t^\dag \comp t / z^\dag \smult z]$ means $t^\dag \comp t = z^\dag \smult z \smult \id$, so Lemma~\ref{lem:daggermono} implies $t=z \bullet s$ for a dagger monic $s$ in $\cat{D}$. 
  But then $[t/z]=[s/1]$ is the image of a dagger monic in $\cat{D}$ under the embedding. 
\end{proof}

An object is \emph{simple} when any monomorphism into it must be 0 or invertible. This requirement on $I$ is stronger than axiom~\eqref{axiom:simple}, which only concerns dagger monomorphisms.

\begin{lemma}\label{lem:simple}
  The tensor unit $I$ in $\cat{D}[\mathcal{D}^{-1}]$ is simple.
\end{lemma}
\begin{proof}
  First, we will show that a monomorphism $[s/z] \colon S \to I$ represents the same subobject of I as $[0/1]$ if and only if $s = 0$. Suppose $[s/z] \colon S \to I$ is a monomorphism in $\cat{D}[\mathcal{D}^{-1}]$. It is always true that $[0/1] \leq [s/z]$ as subobjects, so $[s/z]$ is the minimum subobject if and only if $[s/z] \leq [0/1]$. This happens when there is a morphism $[0,w] \colon S \to 0$ such that $[s/z]\sim[0/1] \comp [0/w]=[0/w]$, that is, $0_{S,I}=z \smult 0 = w \smult s$. Because $w \bullet 0_{S,I} = 0_{S,I}$, by Lemma~\ref{lem:scalarmultiplicationmonic}, this means exactly that $s=0$. Thus $[s/z]=0$ is the minimum subobject if and only if $s=0$.

  Next, we will show that a monomorphism $[s/z] \colon S \to I$ represents the same subobject of $I$ as $[1/1]$ if and only if $s\neq 0$. Similarly to the last paragraph, $[s/z]\leq[\id/1]$ as subobjects, so $[s/z]$ is the maximum subobject if and only if $[\id/1]\leq [s/z]$. This happens when there is a morphism $[t/w] \colon I \to S$ such that $[\id/\id]\sim [s/z] \comp [t/w]=[s \comp t/z \smult w]$, that is, $s \comp t = z \smult w$. We may choose $w=s \comp s^\dag$ and $t=z \smult s^\dag$ to see that this is true as soon as $s \comp s^\dag \neq 0$. This final inequality is equivalent to $s^\dag \neq 0$ because $s$ is monic in this case. 

  Now, either $s=0$ or $s\neq 0$. If $s=0$ then $[s/z]$ is the minimum subobject, and if $s \neq 0$ then $[s/z]$ is the maximum subobject. Therefore $I$ is simple in $\cat{D}[\mathcal{D}^{-1}]$.
\end{proof}

\begin{lemma}\label{lem:separator}
  The tensor unit $I$ is a monoidal separator in $\cat{D}[\mathcal{D}^{-1}]$.
\end{lemma}
\begin{proof}
  Let $[t/z],[t'/z'] \colon H \otimes K \to L$, and suppose $[t/z] \comp [x \otimes y/w] \sim [t'/z'] \comp [x \otimes y/w]$ for all $0 \neq w \colon I \to I$ and $x \colon I \to H$ and $y \colon I \to K$ in $\cat{D}$. Then $(w \smult z \smult t') \comp (x \otimes y)$ is equal to $(w \smult z' \smult t) \comp (x \otimes y)$. By~\eqref{axiom:separator}, $w \smult z \smult t' = w \smult z' \smult t$. Taking $w=1$ gives $z \smult t'=z' \smult t$, that is, $[t/z]\sim[t'/z']$.
\end{proof}

A \emph{dagger biproduct} of $H$ and $K$ in a dagger category with a zero object $0$ is a product $H \stackrel{\pi_1}{\leftarrow} H \oplus K \stackrel{\pi_2}{\rightarrow} K$ such that $\pi_1^\dag$ and $\pi_2^\dag$ are dagger monomorphisms and $\pi_2 \comp \pi_1^\dag = 0$.
 
\begin{lemma}\label{lem:biproducts}
  The category $\cat{D}[\mathcal{D}^{-1}]$ has dagger biproducts, given by $\oplus$ and $0$, and the functor $\cat{D} \to \cat{D}[\mathcal{D}^{-1}]$ is strict monoidal for $\oplus$.
\end{lemma}
\begin{proof}
  For all morphisms $[t/z]\colon H \to K$ and $[t'/z']\colon H' \to K'$, define a morphism $H \oplus H' \to K \oplus K'$ by $[t/z] \oplus [t'/z'] = [z' \bullet t \oplus z \bullet t' / z \bullet z']$. This produces a functor $\oplus \colon \cat{D}[\mathcal{D}^{-1}] \times \cat{D}[\mathcal{D}^{-1}] \to \cat{D}[\mathcal{D}^{-1}]$. The embedding $\cat{D} \to \cat{D}[\mathcal{D}^{-1}]$ induces coherence isomorphisms, making $(\cat{D}[\mathcal{D}^{-1}], \oplus, 0)$ a  symmetric monoidal category. To show that it is cartesian monoidal, it suffices to prove that $K \oplus L$ is the product of $K$ and $L$ with projections $[\inl^\dag/1]$ and $[\inr^\dag/1]$, because $0$ is clearly terminal~\cite{eilenbergkelly:closed}. It then follows that $K \oplus L$ is a dagger biproduct of $K$ and $L$ because $\inl$ and $\inr$ are dagger monomorphisms.

  Axiom~\eqref{axiom:nondegenerate} provides a map $s \colon I \to I \oplus I$ in $\cat{D}$ such that the scalars $x=\inl^\dag \comp s$ and $y=\inr^\dag \comp s$ are nonzero. We will use this to show that $\oplus$ forms a product in $\cat{D}[\mathcal{D}^{-1}]$ with projections $[\inl^\dag/1]$ and $[\inr^\dag/1]$. 
  Let $[r/w] \colon H \to K$ and $[t/z] \colon H \to L$ in $\cat{D}[\mathcal{D}^{-1}]$. 
  Define $q \colon H \to K \oplus L$ to be the following composite:
   \[\begin{tikzcd}[row sep=3mm, column sep=8mm]
    H && {H \oplus H} &&&& {K \oplus L} \\
    {I \otimes H} & {(I \oplus I) \otimes H} & {(I \otimes H) \oplus (I \otimes H)}
    \arrow["\simeq"{marking}, draw=none, from=1-1, to=2-1]
    \arrow["\simeq"{marking}, draw=none, from=2-2, to=2-3]
    \arrow["{[s \otimes \id/1]}"', from=2-1, to=2-2]
    \arrow["\simeq"{marking}, draw=none, from=2-3, to=1-3]
    \arrow["{\big[ z \bullet y \bullet r \oplus w \bullet x \bullet t \big/ w \bullet z \bullet x \bullet y\big]}", from=1-3, to=1-7]
    \arrow["{[p/1]}", dashed, from=1-1, to=1-3]
   \end{tikzcd}\]
  Here $p$ is made up from $s$ and isomorphisms provided by axiom~\eqref{axiom:rig}.
  We will prove that the following diagram commutes in $\cat{D}[\mathcal{D}^{-1}]$:
  \[\begin{tikzpicture}[xscale=4,yscale=2]
    \node (h) at (0,0) {$K$};
    \node (hk) at (1,0) {$K \oplus L$};
    \node (k) at (2,0) {$L$};
    \node (l) at (1,1) {$H$};
    \draw[<-] (h) to node[below]{$[\inl^\dag/1]$} (hk);
    \draw[<-] (k) to node[below]{$[\inr^\dag/1]$} (hk);
    \draw[<-] (h) to node[fill=white]{$[r/w]$} (l);
    \draw[<-] (k) to node[fill=white]{$[t/z]$} (l);
    \draw[<-,dashed] (hk) to node[fill=white]{$q$} (l);
  \end{tikzpicture}\]
  The left triangle commutes because the following diagram commutes in $\cat{D}$:
  \[\begin{tikzcd}[row sep=4mm, column sep=5mm]
    H & {I \otimes H} & {(I \oplus I) \otimes H} & {(I \otimes H) \oplus (I \otimes H)} & {H \oplus H} && {K \oplus L} \\
    & {I \otimes H} & {I \otimes H} & {I \otimes H} \\
    H &&&& H && K
    \arrow["\simeq"{description}, draw=none, from=1-1, to=1-2]
    \arrow["{s \otimes \id}", from=1-2, to=1-3]
    \arrow["\simeq"{description}, draw=none, from=1-3, to=1-4]
    \arrow["\simeq"{description}, draw=none, from=1-4, to=1-5]
    \arrow["{z \smult y \smult r \oplus w \smult x \smult t}", from=1-5, to=1-7, yshift=1mm]
    \arrow["{\inl^\dag}", from=1-7, to=3-7]
    \arrow["{\inl^\dag}", from=1-5, to=3-5]
    \arrow["{z \smult y \smult r}"', from=3-5, to=3-7]
    \arrow[Rightarrow, no head, from=1-1, to=3-1]
    \arrow["{\inl^\dag \otimes \id}", from=1-3, to=2-3]
    \arrow["{\inl^\dag}", from=1-4, to=2-4]
    \arrow["\simeq"{description}, from=2-4, to=3-5]
    \arrow[Rightarrow, no head, from=2-3, to=2-4]
    \arrow["\simeq"{description}, from=3-1, to=2-2]
    \arrow[Rightarrow, no head, from=1-2, to=2-2]
    \arrow["{x \otimes \id}"', from=2-2, to=2-3]
    \arrow["{x \smult \id}"', from=3-1, to=3-5]
  \end{tikzcd}\]
  The right triangle commutes similarly.

  Moreover, the mediating map $q$ is unique.
  For suppose $[t/z],[t'/z'] \colon H \to K \oplus L$ both satisfy $[\inl^\dag/1] \circ [t/z] = [\inl^\dag/1] \circ [t'/z']$ and $[\inr^\dag/1] \circ [t/z] = [\inr^\dag/1] \circ [t'/z']$. Then $\inl^\dag \comp (z \smult t') = \inl^\dag \comp (z' \smult t)$ and $\inr^\dag \comp (z \smult t') = \inr^\dag \comp (z' \smult t)$. Because $\inl^\dag$ and $\inr^\dag$ are jointly epic by axiom~\eqref{axiom:jointlyepic}, we have that $z \smult t'=z' \smult t$. Thus, $[t/z]=[t'/z']$.
\end{proof}

\begin{lemma}\label{lem:equalisers}
  The category $\cat{D}[\mathcal{D}^{-1}]$ has dagger equalisers, and the functor $\cat{D} \to \cat{D}[\mathcal{D}^{-1}]$ preserves them.
\end{lemma}
\begin{proof}
  Let $[t/z],[t'/z'] \colon H \to K$. Let $e \colon E \to H$ be a dagger equaliser of $z' \smult t$ and $z \smult t'$ in $\cat{D}$. Then $[t/z] \comp [e/1] \sim [t'/z'] \comp [e/1]$. 
  Suppose that also $[t/z] \comp [e'/w] \sim [t'/z'] \comp [e'/w]$ for a nonzero scalar $w$ and $e' \colon E' \to H$. Then $z' \smult t \comp w \smult e' = z \smult t' \comp w \smult e'$, and hence there is a unique morphism $m \colon E' \to E$ in $\cat{D}$ with $w \smult e' = e \comp m$.
  Thus, $w \smult w \smult e' = w \smult e \comp m$, that is, $[e'/w] \sim [e/1] \comp [m/ w \smult w]$.
  \[\begin{tikzpicture}[xscale=3.5,yscale=1.5]
    \node (e) at (0,1) {$E$};
    \node (h) at (1,1) {$H$};
    \node (k) at (2,1) {$K$};
    \node (e') at (0,0) {$E'$};
    \draw[->] (e) to node[above]{$[e/1]$} (h);
    \draw[->] ([yshift=.5mm]h.east) to node[above]{$[t/z]$} ([yshift=.5mm]k.west);
    \draw[->] ([yshift=-.5mm]h.east) to node[below]{$[t'/z']$} ([yshift=-.5mm]k.west);
    \draw[->] (e') to node[fill=white]{$[e'/w]$} (h);
    \draw[->,dashed] (e') to node[left]{$[m/w \smult w]$} (e);
  \end{tikzpicture}
  \]
  \[\phantom{[/w \smult w]}
  \begin{tikzpicture}[xscale=3.5,yscale=1.5]
    \node (e) at (0,1) {$E$};
    \node (h) at (1,1) {$H$};
    \node (k) at (2,1) {$K$};
    \node (e') at (0,0) {$E'$};
    \draw[->] (e) to node[above]{$e$} (h);
    \draw[->] ([yshift=.5mm]h.east) to node[above]{$z' \smult t$} ([yshift=.5mm]k.west);
    \draw[->] ([yshift=-.5mm]h.east) to node[below]{$z \smult t'$} ([yshift=-.5mm]k.west);
    \draw[->] (e') to node[fill=white]{$w \smult e'$} (h);
    \draw[->,dashed] (e') to node[left]{$m$} (e);
  \end{tikzpicture}
  \]
  To see that $[m/w \smult w]$ is the unique such morphism, suppose that we also have $[e'/w] \sim [e/1] \comp [n/v]$, that is, $v \smult e' = w \smult e \comp n$.
  We have that $e \comp (v \smult w \smult m) = v \smult w \smult w \smult e' = e \comp (w \smult w \smult w \smult n)$. The morphism $e$ is dagger monic because it is a dagger equalizer, so $v \smult w \smult m = w \smult w \smult w \smult n$. By Lemma~\ref{lem:scalarmultiplicationmonic}, $v \smult m = w \smult w \smult n$, and so $[m/ w \smult w] \sim [n/v]$.
\end{proof}

\begin{lemma}\label{lem:kernels}
  Any dagger monomorphism in $\cat{D}[\mathcal{D}^{-1}]$ is a kernel.  
\end{lemma}
\begin{proof}
  Consider a dagger monomorphism in $\cat{D}[\mathcal{D}^{-1}]$. By Lemma~\ref{lem:dagger}, we may assume that it is of the form $[t/1]$, where $t \colon H \to K$ is a dagger monomorphism in $\cat{D}$.
  Now~\eqref{axiom:kernels} gives $t=\ker(s)$ for some $s \colon K \to L$ in $\cat{D}$. We claim that $[t/1]=\ker([s/1])$ in $\cat{D}[\mathcal{D}^{-1}]$.

  First, clearly $[s/1] \circ [t/1]=0$ because $s \circ t=0$.
  Next suppose that $[s/1] \circ [t'/z'] \sim [0/1] \circ [t'/z']$ for some $[t'/z'] \colon H' \to K$. Then $z' \smult s \comp t'=z' \smult 0$, so $s \comp t'=0$ by Lemma~\ref{lem:scalarmultiplicationmonic}.
  Thus $t'$ factors through $t=\ker(s)$ via some $m \colon H' \to H$.
  Then $[t/1] \comp [m/z'] \sim [t'/z']$ because $z' \smult t \comp m = z' \smult t'$.
  \[\begin{tikzpicture}[xscale=3.5,yscale=1.5]
      \node (h) at (0,1) {$H$};
      \node (k) at (1,1) {$K$};
      \node (l) at (2,1) {$L$};
      \node (h') at (0,0) {$H'$};
      \draw[->] (h) to node[above]{$[t/1]$} (k);
      \draw[->] ([yshift=.5mm]k.east) to node[above]{$[s/1]$} ([yshift=.5mm]l.west);
      \draw[->] ([yshift=-.5mm]k.east) to node[below]{$[0/1]$} ([yshift=-.5mm]l.west);
      \draw[->] (h') to node[fill=white]{$[t'/z']$} (k);
      \draw[->,dashed] (h') to node[left]{$[m/z']$} (h);
    \end{tikzpicture}
  \]
  \[\hspace*{6mm}
    \begin{tikzpicture}[xscale=3.5,yscale=1.5]
      \node (h2) at (1,1) {$H$};
      \node (k) at (2,1) {$K$};
      \node (l) at (3,1) {$L$};
      \node (h') at (1,0) {$H'$};
      \draw[->] (h2) to node[above]{$t$} (k);
      \draw[->] ([yshift=.5mm]k.east) to node[above]{$s$} ([yshift=.5mm]l.west);
      \draw[->] ([yshift=-.5mm]k.east) to node[below]{$0$} ([yshift=-.5mm]l.west);
      \draw[->] (h') to node[fill=white]{$t'$} (k);
      \draw[->,dashed] (h') to node[left]{$m$} (h2);
    \end{tikzpicture}
  \]
  Finally, to see that $[m/z']$ is unique, suppose that $[t/1] \comp [r/w]\sim[t'/z']$. Then $w \smult t \comp m = w \smult t' = z' \smult t \comp r$. It follows that $z' \smult r = w \smult m$ because $t$ is dagger monic, and hence $[m/z']\sim[r/w]$.
\end{proof}

\begin{lemma}\label{lem:colimits}
  The wide subcategory of dagger monomorphisms in $\cat{D}[\mathcal{D}^{-1}]$ has directed colimits.  
\end{lemma}

\begin{proof}
  By Lemma~\ref{lem:dagger}, the wide subcategory of dagger monomorphisms in $\cat{D}[\mathcal{D}^{-1}]$ is isomorphic to the wide subcategory of dagger monomorphisms in $\cat{D}$, so it suffices to show that the wide subcategory of dagger monomorphisms in $\cat{D}$ has directed colimits.
  Let a diagram in $\cat{D}$ be given by dagger monomorphisms $t_{ij} \colon H_i \to H_j$ for all $j \geq i$ in some directed set. By axiom (11), this diagram has a colimit $c_i \colon H_i \to C$ in $\cat{D}$. We show that each of these morphisms is a dagger monomorphism too.

  Fix an index $k$. The diagram consisting of the dagger monomorphisms $t_{ij}$ for all $j \geq i \geq k$ has a colimit $c_i \colon H_i \to C$ for $i \geq k$. But the morphisms $t_{ki}^\dagger$ for $i \geq k$ form another cocone on this diagram because, whenever $j \geq i \geq k$, 
  \[
    t_{ki}^\dagger 
    = t_{ki}^\dagger \circ t_{ij}^\dagger \circ t_{ij} 
    = (t_{ij} \circ t_{ki} )^\dagger \circ t_{ij} 
    = t_{kj}^\dagger \circ t_{ij}\text.
  \]
  Thus, there is a morphism $m_k \colon C \to H_k$ such that $m_k \circ c_i = t_{ki}^\dagger$ for each $i \geq k$, and in particular, $m_k \circ c_k = t_{kk}^\dagger = \id_{H_k}$. Lemma~\ref{lem:daggeriso} now makes $c_k$ a dagger monomorphism.

  Thus the cocone $c_i$ lies within the wide subcategory of dagger monomorphisms in $\cat{D}[\mathcal{D}^{-1}]$.
  Finally, we show that it has the universal property of a colimit within that subcategory.
  Let $d_i\colon H_i \to D$ be any cocone of dagger monomorphisms. 
  This cocone factors through $c_i$ via a morphism $d\colon C \to D$ in $\cat{D}$.
  For all indices $i$ and $j$, there exists $k$ with $i \leq k \geq j$. Now
  \[
    c_j^\dag \circ c_i 
    = t_{jk}^\dag \circ c_k^\dag \circ c_k \circ t_{ik}
    = t_{jk}^\dag \circ d_k^\dag \circ d_k \circ t_{ik}
    = d_j^\dag \circ d_i
    = c_j^\dag \circ d^\dag \circ d \circ c_i\text.
  \]
  Being a colimit cocone in $\cat{D}$, the morphisms $c_i$ are jointly epic, and so the morphisms~$c_j^\dag$ are jointly monic. We conclude that $d^\dag \circ d = \id$.
\end{proof}

Having established that all necessary axioms hold, we can now apply previous results and establish an equivalence with $\cat{Hilb}$. We will describe this equivalence explicitly shortly. 

\begin{proposition}\label{prop:Hilb}
  There is an equivalence of dagger rig categories $\cat{D}[\mathcal{D}^{-1}] \simeq \cat{Hilb}$.
\end{proposition}
\begin{proof}
  Use Lemmas~\ref{lem:monoidal}--\ref{lem:colimits} to apply~\cite[Theorem~8]{heunenkornell:hilbert}. 
\end{proof}

\section{The theorem}\label{sec:theorem}

The previous section showed that if a category $\cat{D}$ satisfies~\eqref{axiom:dagger}--\eqref{axiom:colimits}, then the completion $\cat{C} = \cat{D}[\mathcal{D}^{-1}]$ is equivalent to $\cat{Hilb}_{\mathcal{C}}$ as a dagger rig category. Here $\mathcal{C} = \cat{C}(I, I)$ is an involutive field that is canonically isomorphic either to $\RR$ or to $\CC$ up to a choice of imaginary unit \cite[Proposition 5]{heunenkornell:hilbert}. To be explicit, the equivalence is implemented by the functor $\cat{C}(I, -)\colon \cat{C} \to \cat{Hilb}_{\mathcal C}$. The fact that this is an equivalence of dagger rig categories means the following. First, the functor is full and faithful, and any Hilbert space is unitarily isomorphic to $\cat{C}(I,H)$ for some object $H$ of $\cat{C}$.
Also, the functor preserves the dagger: $\cat{C}(I,f^\dag) = \cat{C}(I,f)^\dag$. Moreover, the functor is strong monoidal for $\otimes$, meaning that there are coherent isomorphisms $\cat{C}(I,H \otimes K) \simeq \cat{C}(I,H) \otimes \cat{C}(I,K)$ and $\cat{C}(I,I) \simeq \mathcal{C}$. Similarly, that the functor is strong monoidal for $\oplus$ means there are coherent isomorphisms $\cat{C}(I,H \oplus K) \simeq \cat{C}(I,H)\oplus \cat{C}(I,K)$ and $\cat{C}(I,0)\simeq 0$. Finally, modulo these coherence morphisms, the functor maps the distributors $H \otimes (K \oplus L) \to (H \otimes K) \oplus (H \otimes L)$ of $\cat{C}$ of axiom~\eqref{axiom:rig} to the canonical distributors in $\cat{Hilb}_{\mathcal C}$.

To complete the proof of the main theorem, in this section we will determine the image of $\cat{D}$ in $\cat{Hilb} = \cat{Hilb}_{\mathcal C}$.
The same proof applies to both the real and the complex cases.
To simplify the presentation, we identify $\cat{D}$ with its canonical image in $\cat{C}$.

First, we characterize $\mathcal{D} \subsetof \mathcal{C}$.

\begin{lemma}\label{lem:disc}
  We have that $\mathcal{D} = \{ z \in \mathcal{C} \mid |z| \leq 1\}$.
\end{lemma}
\begin{proof}
  Let $z \in [0,1] \subsetof \mathcal C$ be an element of the unit interval. The morphism $v \colon I \to I \oplus I$ with components $\sqrt z$ and $\sqrt{1 - z}$ is a dagger monomorphism and thus in $\cat{D}$. It follows that $z = v^\dagger \comp (1 \oplus 0) \comp v$ is also in $\cat{D}$. We conclude that $\mathcal{D} \supseteq \{z \in \mathcal{C} \mid 0 \leq z \leq 1\}$. We also know that $\mathcal{D} \supseteq \{z \in \mathcal{C} \mid |z| = 1\}$ because $\cat{D}$ contains all isometries. Therefore, $\mathcal{D} \supseteq \{z \in \mathcal{C} \mid |z| \leq 1\}$.

  For the reverse inclusion, recall that for any set $A$ we have a directed diagram in $\cat{C}$ that assigns an object $I^R$ to each finite $R \subseteq A$, namely the biproduct of $R$ many copies of $I$. The diagram consists of dagger monomorphisms $i_{R,S}\colon I^R \to I^S$ for finite $R \subseteq S$. For its colimit among the dagger monomorphisms of $\cat{C}$, write $i_{R,A}\colon I^R \to I^A$.

  Because the morphisms of the diagram are dagger monomorphisms, it is also a diagram in $\cat{D}$. Write $j_R \colon I^R \to H$ for the colimit of this diagram in $\cat{D}$. A priori, the objects $H$ and $I^A$ may not be isomorphic. We now specialize to the case $A = \NN$.

  Let $z \in \mathcal{D}$.  Without loss of generality, assume that the objects $I^R$, for finite $R \subseteq \NN$, are constructed using the canonical monoidal product $\oplus$ on $\cat{D}$. Construct a natural transformation $\{t_R \colon I^R \to I^R\}_{R \subseteq \NN}$ in $\cat{D}$ such that $t_{\{n\}}$ is scalar multiplication by $z^n$ for each $n \in \NN$. In the colimit, we obtain a morphism $t\colon H \to H$ such that the following square always commutes:
  \[\begin{tikzcd}
    I^{\{n\}} \arrow{d}[swap]{t_{\{n\}}} \arrow{r}{j_{\{n\}}}
    & H \arrow{d}{T} \\
    I^{\{n\}} \arrow{r}[swap]{j_{\{n\}}}
    & H
  \end{tikzcd}\]
  As in~\cite{heunenkornell:hilbert}, $I^{\{n\}} = I$ for each $n  \in \NN$. Thus, $j_{\{n\}}$ is a vector in the Hilbert space $\cat{C}(I, H)$. The wide subcategory of $\cat{C}$ with dagger monomorphisms is a subcategory of $\cat{D}$, and thus $i_{\{n\},\NN}$ factors through $j_{\{n\}}$. Therefore the vector $j_{\{n\}}$ is nonzero.

  The operator $T = \cat{C}(I, t)$ satisfies $T(j_{\{n\}}) = t \comp j_{\{n\}} = j_{\{n\}} \comp t_{\{n\}}= z^n \smult j_{\{n\}}$. In other words, $j_{\{n\}}$ is an eigenvector of $T$ with eigenvalue $z^n$. Thus, $\|T\| \geq |z^n| = |z|^n$ for all $n \in \NN$. But $T$ is bounded, so we must have $|z| \leq 1$. Therefore, this shows that $\mathcal{D} \subseteq \{z \in \mathcal{C} \mid |z| \leq 1\}$. Altogether, we have equality.
\end{proof}

Our next goal is to show that $\cat{D}$ cannot contain morphisms $T$ of $\cat{Hilb}$ with $\|T\|>1$. To engineer a counterexample, we first examine the situation in $\cat{Con}$.

\begin{example}
  Let $0 < z_1 < z_2 < z_2 < \ldots < 1$ be an increasing sequence in $(0,1]$, and let $z_\infty$ be its supremum. The following is a directed diagram in $\cat{Con}$:
  \[\begin{tikzcd}[column sep=2cm]
    I \arrow{r}{z_1/z_2}
    & I \arrow{r}{z_2/z_3}
    & I \arrow{r}
    & \cdots
  \end{tikzcd}\]
  This diagram admits two cocones, among others: first, a cocone into $I$ whose edges are simply the scalars $z_1, z_2, z_3, \ldots$; second, a cocone into $I$ whose edges are the scalars $z_1/z_\infty, z_2/z_\infty, z_3/z_\infty, \ldots$. This gives the following diagram:
  
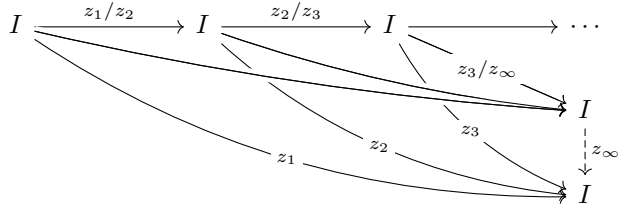
\[\begin{tikzcd}[column sep=2cm]
    I & I & I & \cdots \\
    &&& I \\
    &&& I
    \arrow["{z_1/z_2}", from=1-1, to=1-2]
    \arrow["{z_2/z_3}", from=1-2, to=1-3]
    \arrow["{z_\infty}", dashed, from=2-4, to=3-4]
    \arrow[from=1-3, to=1-4]
    \arrow["{z_1}"{description}, curve={height=24pt}, from=1-1, to=3-4]
    \arrow["{z_2}"{description}, curve={height=18pt}, from=1-2, to=3-4]
    \arrow["{z_3}"{description}, curve={height=12pt}, from=1-3, to=3-4]
    \arrow[halo, "{z_3/z_\infty}"{description}, from=1-3, to=2-4]
    \arrow[halo, curve={height=6pt}, from=1-2, to=2-4]
    \arrow[halo, curve={height=6pt}, from=1-1, to=2-4]
    \arrow["{z_3/z_\infty}"{description}, from=1-3, to=2-4]
    \arrow[curve={height=6pt}, from=1-2, to=2-4]
    \arrow[curve={height=6pt}, from=1-1, to=2-4]
  \end{tikzcd}\]
  It is routine to verify that $z_1/z_\infty, z_2/z_\infty, z_3/z_\infty, \ldots$ is a colimiting cocone.
\end{example}

The following lemma demonstrates the same phenomenon in $\cat{D}$.

\begin{lemma}\label{lem:colimiting}
  The functor $\cat{C}(I, -)\colon \cat{C} \to \cat{Hilb}$ restricts to a functor $\cat{D} \to \cat{Con}$. Furthermore, for all objects $H$ and all sequences $0 < z_1 < z_2 < \ldots$ with $\sup z_n = 1$, the following is a colimiting cocone in $\cat{D}$:
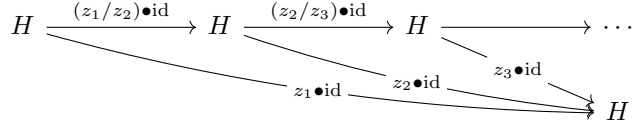
  \[\begin{tikzcd}[column sep=2cm]
    H & H & H & \cdots \\
    &&& H
    \arrow["{(z_1/z_2)\smult\id}", from=1-1, to=1-2]
    \arrow["{(z_2/z_3)\smult\id}", from=1-2, to=1-3]
    \arrow[from=1-3, to=1-4]
    \arrow["{z_3 \smult \id}"{description}, from=1-3, to=2-4]
    \arrow["{z_2\smult\id}"{description}, curve={height=6pt}, from=1-2, to=2-4]
    \arrow["{z_1\smult\id}"{description}, curve={height=10pt}, from=1-1, to=2-4]
  \end{tikzcd}\]
\end{lemma}
\begin{proof}
  Let $t \colon H \to K$ be a morphism in $\cat{D}$. Assume that $\|\cat{C}(I, t)\| > 1$, and let $x \in \cat{C}(I, H)$ be such that $\|x\| = 1$ and $\|t \comp x\| > 1$. Now $x^\dagger \comp x = \| x \|^2 = 1$, so $x$ is dagger monic and hence in $\cat{D}$. Furthermore $x^\dagger \comp t^\dagger \comp t \comp x  = \| t \comp x \|^2 > 1$. But $x^\dagger \comp t^\dagger \comp t \comp x$ is a scalar in $\cat{D}$, contradicting Lemma~\ref{lem:disc}.
  We conclude that $\|\cat{C}(I, t)\| \leq 1$, and that $\cat{C}(I, -)$ restricts to a functor $\cat{D} \to \cat{Con}$.

  Let $0 < z_1 < z_2 < \ldots$ be an increasing sequence of numbers with supremum $1$. The morphisms $(z_n/z_{n+1}) \smult \id$ are all in $\cat{D}$ by Lemma~\ref{lem:disc}. They form a directed diagram, which has a colimiting cocone $t_n \colon H \to K$. Now $z_n \smult \id \colon H \to H$ forms another cocone. Thus there is a unique mediating morphism $s \colon K \to H$ in $\cat{D}$:
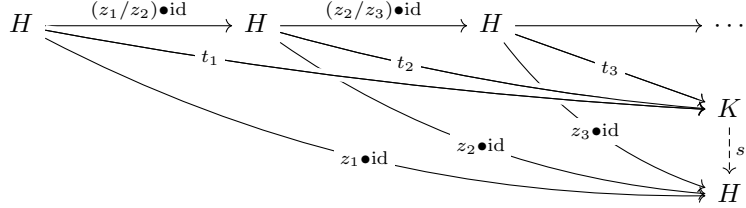
  \[\begin{tikzcd}[column sep=25mm]
    H & H & H & \cdots \\
    &&& K \\
    &&& H
    \arrow["{(z_1/z_2) \smult \id}", from=1-1, to=1-2]
    \arrow["{(z_2/z_3) \smult \id}", from=1-2, to=1-3]
    \arrow["{s}", dashed, from=2-4, to=3-4]
    \arrow[from=1-3, to=1-4]
    \arrow["{z_1 \smult \id}"{description}, curve={height=24pt}, from=1-1, to=3-4]
    \arrow["{z_2 \smult \id}"{description}, curve={height=18pt}, from=1-2, to=3-4]
    \arrow["{z_3 \smult \id}"{description}, curve={height=12pt}, from=1-3, to=3-4]
    \arrow[halo, "t_3"{description}, from=1-3, to=2-4]
    \arrow[halo, curve={height=6pt}, from=1-2, to=2-4]
    \arrow[halo, curve={height=6pt}, from=1-1, to=2-4]
    \arrow["t_3"{description}, from=1-3, to=2-4]
    \arrow["t_2"{description}, pos=.3, curve={height=6pt}, from=1-2, to=2-4]
    \arrow["t_1"{description}, pos=.25, curve={height=6pt}, from=1-1, to=2-4]
  \end{tikzcd}\]
  The functor $\cat{C}(I, -)\colon \cat{D} \to \cat{Con}$ maps the cocone $z_n \smult \id \colon H \to H$ to a colimiting cocone in $\cat{Con}$. Thus there is a mediating bounded operator $T \colon \cat{C}(I, H) \to \cat{C}(I, K)$:
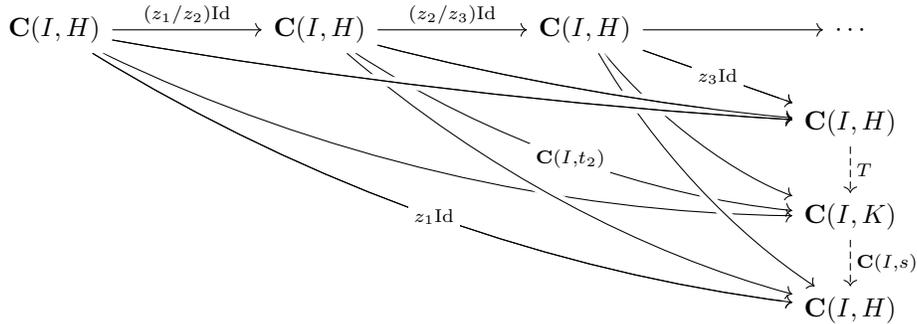
  \[\begin{tikzcd}[column sep=20mm]
    \cat{C}(I,H) & \cat{C}(I,H) & \cat{C}(I,H) & \cdots \\
    &&& \cat{C}(I,H) \\
    &&& \cat{C}(I,K) \\
    &&& \cat{C}(I,H)
    \arrow["{(z_1/z_2) \mathrm{Id}}", from=1-1, to=1-2]
    \arrow["{(z_2/z_3) \mathrm{Id}}", from=1-2, to=1-3]
    \arrow["{T}", dashed, from=2-4, to=3-4]
    \arrow["{\cat{C}(I,s)}", dashed, from=3-4, to=4-4]
    \arrow[from=1-3, to=1-4]
    \arrow[curve={height=24pt}, from=1-1, to=3-4]
    \arrow["{\cat{C}(I,t_2)}"{description}, curve={height=18pt}, from=1-2, to=3-4]
    \arrow[curve={height=12pt}, from=1-3, to=3-4]
    \arrow[halo, "{z_1\mathrm{Id}}"{description}, curve={height=24pt}, from=1-1, to=4-4]
    \arrow[halo, curve={height=18pt}, from=1-2, to=4-4]
    \arrow[halo, curve={height=12pt}, from=1-3, to=4-4]
    \arrow[halo, "{z_3\mathrm{Id}}"{description}, from=1-3, to=2-4]
    \arrow[halo, curve={height=6pt}, from=1-2, to=2-4]
    \arrow[halo, curve={height=6pt}, from=1-1, to=2-4]
    \arrow["{z_3\mathrm{Id}}"{description}, from=1-3, to=2-4]
    \arrow[curve={height=6pt}, from=1-2, to=2-4]
    \arrow[curve={height=6pt}, from=1-1, to=2-4]
    \arrow["{z_1\mathrm{Id}}"{description}, curve={height=24pt}, from=1-1, to=4-4]
  \end{tikzcd}\]
  The universal property now implies that $\cat{C}(I, s) \comp T$ is the identity on $\cat{C}(I, H)$. Because both $\cat{C}(I, s)$ and $T$ are contractions, $T$ must be an isometry. It follows that there is a dagger monomorphism $t \colon H \to K$ in $\cat{C}$, and hence in $\cat{D}$, with $\cat{C}(I, t) = T$. 

  Thus, the morphisms $s$ and $t$ in $\cat{D}$ satisfy $\cat{C}(I, s \comp t) = \cat{C}(I, s) \comp \cat{C}(I, t) = \cat{C}(I, \id)$. Because the functor $\cat{C}(I, -)$ is faithful, $s \comp t = \id$. In the same way, we find that $\cat{C}(I, t_n) = \cat{C}(I, t) \comp (z_n \id) = \cat{C}(I, z_n \smult t)$, so $t_n = t \comp (z_n \smult \id) = t \comp s \comp t_n$, concluding that $t \comp s = \id$. Since $t$ is dagger monic, both $s$ and $t$ are dagger isomorphisms that are inverse to each other. Thus the cocone $z_n \smult \id \colon H \to H$ is colimiting in $\cat{D}$.
\end{proof}

Finally, we show that the image of $\cat{D}$ in $\cat{Hilb}$ contains all contractions.

\begin{lemma}\label{lem:contractions}
  For all objects $H$ and $K$, $\cat{D}(H, K) = \{t \in \cat{C}(H, K) \mid \|\cat{C}(I, t) \| \leq 1\}$.
\end{lemma}
\begin{proof}
  It remains to show that morphisms $t$ of $\cat{C}$ with $\|\cat{C}(I, t)\| \leq 1$ are in $\cat{D}$. First, suppose $t \colon H \to H$ in $\cat{C}$ satisfies $\|\cat{C}(I, t)\|< 1$. It follows that $\cat{C}(I,t)$ is a convex combination of unitary operators on the Hilbert space $\cat{C}(I, H)$: there exist unitary operators $U_1, \ldots, U_n$ and positive real numbers $z_1, \ldots, z_n$ such that $z_1 + \cdots + z_n = 1$ and $z_1 U_1 + \cdots + z_n U_n = \cat{C}(I,t)$ \cite[Theorem~13]{navarro:real}. The operator \begin{align*} W \colon \cat{C}(I, H) {}& \to \cat{C}(I, H) \oplus \ldots \oplus \cat{C}(I, H) \\ x {}& \mapsto (\sqrt{z_1} x, \cdots, \sqrt{z_n} x)\end{align*} then satisfies $W^\dagger \comp (T_1 \oplus \cdots \oplus T_n) \comp W = z_1 T_1 + \cdots + z_n T_n$ for all operators $T_1, \ldots, T_n$ on $\cat{C}(I, H)$. Furthermore, $W$ is an isometry. Now~\cite[Theorem~8]{heunenkornell:hilbert} provides dagger isomorphisms $u_1 \ldots, u_n \colon H \to H$ in $\cat{C}$ such that $U_i = \cat{C}(I, u_i)$ for each $i$ and a dagger isometry $w \colon H \to H \oplus \cdots \oplus H$ in $\cat{C}$ such that $W = \cat{C}(I, w)$. By Lemma~\ref{lem:dagger}, $u_1, \ldots, u_n$, and $w$ are morphisms in $\cat{D}$, and thus, $w^\dagger \circ (u_1 \oplus \ldots \oplus u_n)\circ w$ is a morphism in $\cat{D}$. Also,
  $
  \cat{C}(I, w^\dagger \circ (u_1 \oplus \ldots \oplus u_n)\circ w) = W^\dagger \comp (U_1 \oplus \cdots \oplus U_n) \comp W = z_1 U_1 + \cdots + z_n U_n = \cat{C}(I,t),
  $
  so $w^\dagger \circ (u_1 \oplus \ldots \oplus u_n)\circ w = t$. We conclude that $t$ is in $\cat{D}$ and, more generally, that any endomorphism $t$ of $\cat{C}$ with $\|\cat{C}(I, t)\| < 1$ is in $\cat{D}$.

  Next, suppose $t \colon H \to K$ in $\cat{C}$ satisfies $\|\cat{C}(I, t)\|<1$. 
  Using polar decomposition, any operator $T$ in $\cat{Hilb}$ factors as $T = U \comp V^\dagger \comp S$, where $U$ and $V$ are isometries and $S$ is a bounded operator on the domain of $T$ such that $\|S\| = \|T\|$. 
  Because $\cat{C}(I, -) \colon \cat{C} \to \cat{Hilb}$ is an equivalence of monoidal dagger categories, $t$ factors as $ t = u \comp v^\dagger \comp s$, where $u$ and $v$ are dagger monomorphisms and $s$ is an endomorphism such that $\|\cat{C}(I, s) \| < 1$. Thus, $t$ is in $\cat{D}$. Therefore, any morphism $t$ in $\cat{C}$ with $\|\cat{C}(I, t)\| <1$ is in $\cat{D}$.

  Finally, suppose $t \colon H \to K$ in $\cat{C}$ satisfies $\|\cat{C}(I, t)\| = 1$. Let $0 < z_1 < z_2 < \ldots$ be a sequence of numbers with supremum $1$. For each $0 <z < 1$, we calculate that $\|\cat{C}(I, z \smult t) \| = \| z \cat{C}(I, t) \| = z \|\cat{C}(I,t)\| < 1$; hence $z \smult t \in \cat{D}(H, K)$. Therefore, the following diagram lies entirely in $\cat{D}$:
  \[\begin{tikzcd}[column sep=25mm]
    H & H & H & \cdots \\
    &&& H \\
    &&& K
    \arrow["{(z_1/z_2) \smult \id}", from=1-1, to=1-2]
    \arrow["{(z_2/z_3) \smult \id}", from=1-2, to=1-3]
    \arrow["{s}", dashed, from=2-4, to=3-4]
    \arrow[from=1-3, to=1-4]
    \arrow["{z_1 \smult t}"{description}, curve={height=24pt}, from=1-1, to=3-4],
    \arrow["{z_2 \smult t}"{description}, curve={height=18pt}, from=1-2, to=3-4]
    \arrow["{z_3 \smult t}"{description}, curve={height=12pt}, from=1-3, to=3-4]
    \arrow[halo, "z_3 \smult \id"{description}, from=1-3, to=2-4]
    \arrow[halo, "z_2 \smult \id"{description}, pos=.3, curve={height=6pt}, from=1-2, to=2-4]
    \arrow[halo, "z_1 \smult \id"{description}, pos=.25, curve={height=6pt}, from=1-1, to=2-4]
    \arrow["z_3 \smult \id"{description}, from=1-3, to=2-4]
    \arrow["z_2 \smult \id"{description}, pos=.3, curve={height=6pt}, from=1-2, to=2-4]
    \arrow["z_1 \smult \id"{description}, pos=.25, curve={height=6pt}, from=1-1, to=2-4]
  \end{tikzcd}\]
  Lemma~\ref{lem:colimiting} provides a unique morphism $s \in \cat{D}(H, K) \subseteq \cat{C}(H, K)$ making this diagram commute. Thus, $z_1 \smult s = s \comp (z_1 \smult \id) = z_1 \smult t$. Since $z_1 \neq 0$, we conclude by Lemma~\ref{lem:scalarmultiplicationmonic} that $t=s$. Hence $t$ is in $\cat{D}$. Therefore, any morphism $t$ of $\cat{C}$ with $\|\cat{C}(I, t)\| \leq 1$ is in $\cat{D}$.
\end{proof}

We have arrived at the main theorem of this article.

\begin{theorem}\label{thm:main}
  A dagger rig category $\cat{D}$ satisfying axioms~\eqref{axiom:dagger}--\eqref{axiom:colimits} is equivalent to $\cat{Con}_\RR$ or $\cat{Con}_\CC$, and the equivalence preserves daggers and is strong symmetric monoidal for both $\otimes$ and $\oplus$ and preserves the distributors of axiom~\eqref{axiom:rig}.
\end{theorem}
\begin{proof}
  Combine Proposition~\ref{prop:Hilb} and Lemma~\ref{lem:contractions}.
\end{proof}

All of the isomorphisms of axiom~\eqref{axiom:rig} are preserved, making any dagger rig category satisfying axioms~\eqref{axiom:dagger}--\eqref{axiom:colimits} equivalent to $\cat{Con}_\RR$ or $\cat{Con}_\CC$ as a dagger rig category. The equivalence of Theorem~\ref{thm:main} is a weak equivalence, i.e., a full, faithful, and essentially surjective functor $\cat{D} \to \cat{Con}$. In fact, it is essentially surjective in the stronger sense that every object of $\cat{Con}$ is dagger isomorphic to one in its image by Lemma~\ref{lem:daggeriso}.

Let $\cat{D}$ be a category. If there exist a contravariant endofunctor $\dagger$ and symmetric monoidal structures $(\otimes, I)$ and $(\oplus, 0)$ that satisfy axioms~\eqref{axiom:dagger}--\eqref{axiom:colimits}, then there exists an equivalence $\cat{D} \to \cat{Con}$, as an immediate consequence of Theorem~\ref{thm:main}. The converse also holds.
For any equivalence $\cat{D} \to \cat{Con}$, there exist a dagger and dagger symmetric monoidal structures $(\otimes, I)$ and $(\oplus, 0)$ on $\cat{D}$ such that the equivalence preserves all of these structures; we combine~\cite[Lemma~V.1]{vicary:complex}, \cite[Corollary~3.1.6]{karvonen:dagger}, and~\cite[Theorem~5]{jacobsmandemaker:coreflections} and use the axiom of global choice.

Thus, a category $\cat{D}$ is equivalent to $\cat{Con}_\RR$ or $\cat{Con}_\CC$ if and only if there exist a dagger and symmetric monoidal structures $(\otimes, I)$ and $(\oplus, 0)$ on $\cat{D}$ satisfying axioms~\eqref{axiom:dagger}--\eqref{axiom:colimits}. In this way, these axioms serve to characterize $\cat{Con}_\RR$ and $\cat{Con}_\CC$ as categories with no additional structure.

\section*{Acknowledgments}

We thank Matthew Di Meglio and Martti Karvonen for careful reading and suggestions.

\bibliographystyle{plain}
\bibliography{bibliography}

\end{document}